\documentclass[12pt]{amsart}
\usepackage{amssymb}
\usepackage{hyperref} 
\usepackage{dsfont,color,xcolor}				
\usepackage{mathrsfs}
\usepackage[normalem]{ulem}
\usepackage{todonotes}

\usepackage[all]{xy}
\newtheorem{theorem}{Theorem}[section]
\newtheorem{proposition}[theorem]{Proposition}
\newtheorem{coro}[theorem]{Corollary}
\newtheorem{lemma}[theorem]{Lemma}
\newtheorem{rem}[theorem]{Remark}
\newtheorem{definition}[theorem]{Definition}

\newtheorem*{assum-carre}{Carr\'e du champ identity}
\newtheorem*{assum-LGamma}{Assumptions on $L$ and $\Gamma$}

\renewcommand{\epsilon}{\varepsilon}
\newcommand{\eps}{\epsilon}

\newcommand\C{\mathbb{C}}
\newcommand\N{\mathbb{N}}
\newcommand\R{\mathbb{R}}
\newcommand\calS{\mathcal{S}}

\newcommand{\scrC}{\ensuremath{\mathscr{C}}}

\newcommand{\dive}{\textrm{div}}

\DeclareMathOperator{\loc}{loc}
\newcommand{\Eins}{\ensuremath{\mathds{1}}}

\def\Xint#1{\mathchoice
   {\XXint\displaystyle\textstyle{#1}}%
   {\XXint\textstyle\scriptstyle{#1}}%
   {\XXint\scriptstyle\scriptscriptstyle{#1}}%
   {\XXint\scriptscriptstyle\scriptscriptstyle{#1}}%
   \!\int}
\def\XXint#1#2#3{{\setbox0=\hbox{$#1{#2#3}{\int}$}
     \vcenter{\hbox{$#2#3$}}\kern-.5\wd0}}

\def\aver#1{\Xint-_{#1}}

\newcommand\restr[2]{{
  \left.\kern-\nulldelimiterspace 
  #1 
  \vphantom{\big|} 
  \right|_{#2} 
  }}

\numberwithin{theorem}{section}
\numberwithin{equation}{section}

\usepackage{color}

\definecolor{gr}{rgb}   {0.,   0.8,   0. } 
\definecolor{bl}{rgb}   {0.,   0.5,   1. } 
\definecolor{mg}{rgb}   {0.7,  0.,    0.7}

\title[Sobolev algebras]{Sobolev algebras through a ``carr\'e du champ'' identity}

\author{Fr\'ed\'eric Bernicot}
\address{Fr\'ed\'eric Bernicot, CNRS - Universit\'e de Nantes, Laboratoire Jean Leray, 2 rue de la Houssini\`ere, 44322 Nantes cedex 3.  France}
\email{frederic.bernicot@univ-nantes.fr}

\author{Dorothee Frey}
\address{Dorothee Frey, Delft Institute of Applied Mathematics, Delft University of Technology, P.O. Box 5031, 2600 GA Delft, The Netherlands}
\email{d.frey@tudelft.nl}

\thanks{FB's research was supported by the ERC project FAnFArE no. 637510 and
by the ANR project HAB no. ANR-12-BS01-0013}

\date{\today}

\let\oldtocsection=\tocsection
\let\oldtocsubsection=\tocsubsection
\let\oldtocsubsubsection=\tocsubsubsection

\renewcommand{\tocsection}[2]{\hspace{0em}\oldtocsection{#1}{#2}}
\renewcommand{\tocsubsection}[2]{\hspace{1em}\oldtocsubsection{#1}{#2}}
\renewcommand{\tocsubsubsection}[2]{\hspace{2em}\oldtocsubsubsection{#1}{#2}}

\begin{document}
\begin{abstract} We consider abstract Sobolev spaces of Bessel-type associated with an operator. In this work, we pursue the study of algebra properties of such functional spaces through the corresponding semigroup. As a follow-up of \cite{BCF2}, we show that under the extra property of a ``carr\'e du champ identity'', this algebra property holds in a wider range than previously shown. 
\end{abstract}

\maketitle


\section{Introduction}

\subsection{Setting}
\label{sec11}

Let $(M,d)$ be a locally compact separable metric space, equipped with a Borel measure $\mu$, finite on compact sets and strictly positive on any non-empty open set. For $\Omega$ a measurable subset of $M$, we shall  denote $\mu\left(\Omega\right)$ by $\left|\Omega\right|$.  For all $x \in M$ and all $r>0$, denote by $B(x,r)$ the open ball for the metric $d$ with centre $x$ and radius $r$, and  by $V(x,r)$ its measure $|B(x,r)|$.  For a ball $B$ of radius $r$ and a real $\lambda>0$, denote by $\lambda B$   the ball concentric  with $B$ and with radius $\lambda r$. We shall sometimes denote by $r(B)$ the radius of a ball $B$. We will use $u\lesssim v$ to say that there exists a constant $C$ (independent of the important parameters) such that $u\leq Cv$, and $u\simeq v$ to say that $u\lesssim v$ and $v\lesssim u$. Moreover, for $\Omega\subset M$ a subset of finite and non-vanishing measure and $f\in L^1_{loc}(M,\mu)$, $\aver{\Omega} f \, d\mu=\frac{1}{|\Omega|} \int f \, d\mu$ denotes the average of $f$ on $\Omega$. 

From now on, we assume that $(M,d,\mu)$ is a doubling metric measure space, which means that the measure $\mu$ satisfies the doubling property, that is
  \begin{equation}\label{d}\tag{$V\!D$}
     V(x,2r)\lesssim  V(x,r),\quad \forall~x \in M,~r > 0.
    \end{equation}
As a consequence, there exists  $\nu>0$  such that
     \begin{equation*}\label{dnu}\tag{$V\!D_\nu$}
      V(x,r)\lesssim \left(\frac{r}{s}\right)^{\nu} V(x,s),\quad \forall~r \ge s>0,~ x \in M.
    \end{equation*}

We then consider an unbounded operator $L$ on $L^2(M,\mu)$ as well as an 'abstract' notion of gradient operator $\Gamma$ under the following assumptions:

\begin{assum-LGamma}
\begin{itemize}
\item 
Assume that  $L$ is an injective, $\omega$-accretive operator with dense domain $\mathcal D \subset L^2(M,\mu)$, where $0 \leq \omega < \pi/2$.
Assume that there exists a bilinear operator $\Gamma$, with domain ${\mathcal F}^2$ for some subset ${\mathcal F}$ of $L^2(M,\mu)$, with $\mathcal{D} \subset \mathcal{F}$. 
\item For every $f\in {\mathcal F}$, we set $\Gamma(f):=|\Gamma(f,f)|^{1/2}$ and assume that $\Gamma$ satisfies the inequality
\begin{equation} \left|\Gamma(f,g)\right|\leq \Gamma(f) \Gamma(g), \qquad \forall f,g, \in \mathcal{F}.
\label{eq:gamma}
\end{equation}
Moreover, assume that 
 \begin{equation}\tag{$R_2$}
\| \Gamma f \|_2 \lesssim \|L^{1/2} f\|_2, \qquad \forall f\in {\mathcal D}.
 \label{R2}
\end{equation}
\item Assume that the semigroup $(e^{-tL})_{t>0}$ admits a kernel representation with a kernel $p_t$ satisfying the upper Gaussian pointwise estimates
\begin{equation}\tag{$U\!E$}
\left|p_{t}(x,y)\right|\lesssim
\frac{1}{V(x,\sqrt{t})}\exp
\left(-\frac{d^{2}(x,y)}{Ct}\right), \quad \forall~t>0,\, \mbox{a.e. }x,y\in
 M.\label{UE}
\end{equation}

\item Assume that the semigroup $(e^{-tL})_{t>0}$ and its gradient satisfy $L^2$ Davies-Gaffney estimates, which means that 
for every $r>0$ and all balls $B_1$,$B_2$ of radius $r$
\begin{equation} \tag{$DG$}
\| e^{-r^2L} \|_{L^2(B_1) \to L^2(B_2)} + \| r\Gamma e^{-r^2L} \|_{L^2(B_1) \to L^2(B_2)} \lesssim e^{-c \frac{d^2(B_1,B_2)}{r^2}}.
\label{eq:DG}
\end{equation}
\end{itemize}
\end{assum-LGamma}

By our assumptions, $(e^{-tL})_{t>0}$ is bounded analytic on $L^p(M,\mu)$ for $p \in (1,\infty)$ and uniformly bounded on $L^p(M,\mu)$ for $p \in [1,\infty]$, see \cite[Corollary 1.5]{BK2}. 
Note that \eqref{eq:DG} for the semigroup is a consequence of \eqref{UE}. By analyticity of the semigroup, the property \eqref{UE}, and thus also \eqref{eq:DG}, extends to the collections $((tL)^ne^{-tL})_{t>0}$ for every integer $n\geq 0$. The operator $\Gamma$ is a sublinear operator, acting like the length of the gradient on a Riemannian manifold. 

We also assume that $\Gamma$ and $L$ are related by a weak version of a ``carr\'e du champ identity'':

\begin{assum-carre}
 Assume that $\Gamma$ and $L$ satisfy the following: for every $t>0$
 and all functions $f,g\in L^\infty(M,\mu) \cap {\mathcal D}$ 
\begin{equation}
		e^{-tL}L(fg) = e^{-tL}\big[Lf \cdot g\big] +e^{-tL}\big[f \cdot Lg\big] - 2 e^{-tL} \Gamma(f,g).
\label{carre}
\end{equation}
This equality can be viewed in $L^2_{\loc}(M,\mu)$, since for  functions $f,g$ chosen as above, we know that  $\Gamma(f,g) \in L^1(M,\mu)$ and so the LHS and RHS are both locally in $L^2(M,\mu)$ due to \eqref{UE}.
\end{assum-carre}

\begin{rem}
\begin{itemize}
\item Note that the full carr\'e du champ identity, which is
\begin{equation}
		L(fg) = Lf \cdot g +f \cdot Lg - 2\Gamma(f,g),
\label{eq:strong}
\end{equation}
is stronger than the previous assumption. It is not clear on which set of functions such an identity may be assumed. 

\item Let us emphasise that the proofs developed in the next sections do not really require the exact identity \eqref{carre}. It would be sufficient to only assume the following inequality: for every $t>0$
 and all functions $f,g\in L^\infty(M,\mu) \cap {\mathcal D}$
\begin{equation}
		\left|e^{-tL}L(fg) - \left(e^{-tL}\big[Lf \cdot g\big] +e^{-tL}\big[f \cdot Lg\big]\right)\right| \lesssim \left|e^{-tL} \Gamma(f,g) \right|.
\label{carre-w}
\end{equation}
\end{itemize}
\end{rem}

We will assume the above throughout the paper. We abbreviate the setting with $(M,\mu,\Gamma,L)$.

\subsection{The algebra property}

Following up on \cite{BCF2}, we aim to prove that the (Bessel-type) Sobolev spaces satisfy an algebra property under our assumptions. Such property is very well understood in the Euclidean space and goes back to initial works by Strichartz \cite{St}, Kato and Ponce \cite{KP}, and then Coifman and Meyer \cite{cm,Meyer} using the paraproduct decomopsition. We refer the reader to \cite{BCF2} and references therein for a more complete review of the literature on this topic.
This algebra property and the corresponding Leibniz rule is crucial in order to study nonlinear PDEs.

In this current work, we are going to describe how the ``carr\'e du champ'' property allows us to improve the main results of \cite{BCF2}. Indeed, the carr\'e du champ identity combined with \eqref{eq:gamma} encodes a kind of algebra property at the order of regularity $1$, since the operator $\Gamma$ (resp. $L$) is implicitly an operator of order $1$ (resp. $2$).

\medskip

Let us first give a rigorous sense to what we mean by the Algebra property for Sobolev spaces. We follow the approach of \cite{BCF2}. Denote by $\mathcal{C}_0(M)$ the space of continuous functions on $M$ which vanish at infinity, and $\mathfrak{C}:=\mathcal{C}_0(M) \cap \mathcal{F}$. 
We define $\dot L^p_\alpha(M,L,\mu)\cap L^\infty(M,\mu)$ as the completion of
$$  \left\{f \in \mathfrak{C},\ L^{\alpha / 2}f\in L^p(M,\mu) \right\} $$ with respect to the norm  $\left\|L^{\alpha/2}f\right\|_p+\left\|f\right\|_\infty$. We denote in the sequel $\|L^{\alpha/2} f\|_p$ by $\|f\|_{p,\alpha}$.

Let us recall our definition of the algebra property $A(p,\alpha)$ from \cite[Definition 1.1]{BCF2}.

\begin{definition} For $\alpha>0$ and $p\in(1,\infty)$ we say that property $A(p,\alpha)$ holds if: 
\begin{itemize}
\item the space $\dot{L}^p_{\alpha}(M,L,\mu) \cap L^{\infty}(M,\mu)$ is an algebra for the pointwise product;
\item and the Leibniz rule inequality is valid:
$$ \| fg \|_{p,\alpha} \lesssim \|f\|_{p,\alpha} \|g\|_{\infty} +  \|f\|_{\infty} \|g\|_{p,\alpha}, \quad \forall\,f,g\in \dot{L}^p_{\alpha}(M,L,\mu)\cap L^\infty(M,\mu).$$
\end{itemize}
\end{definition}

\subsection{Main result}

For $p \in [1,\infty]$, we say that the semigroup satisfies gradient bounds \eqref{Gp} if 
\begin{equation} \label{Gp}
\sup_{t>0} \|\sqrt{t}\Gamma e^{-tL} \|_{p\to p} <\infty \tag{$G_p$}.
\end{equation}
Let us observe that by \eqref{R2} and \eqref{UE}, it is classical that our previous assumptions already imply \eqref{Gp} for $p\in(1,2]$.

Our main result reads as follows:

\begin{theorem} \label{thm:main}
 Let $(M,\mu,\Gamma,L)$ as in Subsection \ref{sec11} with a homogeneous dimension $\nu>2$.  Assume in addition 
  $(G_{p_0})$ for some $p_0 \in [2,\nu)$. Then 
 $A(p,\alpha)$ holds for every $p \in (1,p_0)$ with  $\alpha \in (0,1)$, and for every $p \in (p_0,\infty)$ with $0<\alpha<\frac{p_0}{p}$.
\end{theorem}

The condition $p_0<\nu$ is not relevant and not used, but for $p_0>\nu$ the result was already obtained in \cite{BCF2} in a more general framework. That is why we restrict our attention here to the range $2\leq p_0<\nu$.

We use a slightly different decomposition of the product than in \cite{BCF2}. Indeed in \cite{BCF2}, the product of two functions was decomposed into two paraproducts. Here, we decompose it into three terms (two paraproducts and a 'resonant part'). The two paraproducts are completely uncritical, whereas the third one carries the most subtle information encoded in the resonances. The carr\'e du champ identity now allows us to handle this third part in a better way. This allows to improve over \cite{BCF2} in the case $p>2$.

\begin{proof} The theorem will be proved in the following sections. The proof goes through the use of Stein's complex interpolation between the two endpoints $(\alpha,p)=(1,p_0)$ and $(\alpha,p)=(0,\infty)$. \\
The case $p\in (1,p_0)$ is obtained as the combination of the paraproduct decompositions \eqref{eq:dec1} and \eqref{eq:decomp-carre} with the boundedness results of Propositions \ref{prop:errorterms}, \ref{prop:error-2} and \ref{prop:sp-para}. The case $p\in [p_0,\infty)$ is shown by combining the paraproduct decompositions \eqref{eq:dec1} and \eqref{eq:decomp-carre} with the boundedness results of Propositions \ref{prop:errorterms} and \ref{prop:inter}.
\end{proof}

\subsection{Comparison to previous results and examples}
Let us compare this result with what we have previously obtained in \cite[Theorem 1.5]{BCF2}. First, let us mention that even if \cite{BCF2} was written in the setting of a Dirichlet form (which is a particular case of our current setting here), all of the results in \cite{BCF2} can be described in our present setting, without assuming the 'carr\'e du champ' identity, with identical proofs. The extra main property used in \cite{BCF2} (instead of \eqref{carre}) is the following inequality
\begin{equation} \left| \int Lf \cdot g \, d\mu \right| \lesssim \int \Gamma f \cdot \Gamma g\, d\mu \label{carre2}
\end{equation}
for all functions $f,g\in {\mathcal F}$.

\smallskip

Let us now compare our result with the one of \cite{BCF2}:
\begin{itemize}
\item The two approaches rely on the same framework given by a 'gradient' operator $\Gamma$ satisfying a Leibniz rule and a semigroup $(e^{-tL})_{t>0}$. The main difference is that \cite{BCF2} requires \eqref{carre2}, whereas here we assume \eqref{carre} or in fact the weaker version \eqref{carre-w}. We first observe that in the case of a self-adjoint and conservative operator $L$, then by integrating \eqref{carre-w} implies exactly \eqref{carre2}. So our current assumption is stronger than the one used in \cite{BCF2} and corresponds to a pointwise version; it is therefore natural that we are able to obtain a wider range of exponents. To be more precise, for $p>p_0$ we improve the range $\alpha \in (0, 1-\nu(\frac{1}{p_0}-\frac{1}{p}))$ (obtained in \cite{BCF2}) to $\alpha \in (0,\frac{p_0}{p})$.

\item Moreover, we only detail the proofs of \cite{BCF2} and of the current work in the setting where the semigroup is supposed to satisfy \eqref{UE}, which corresponds to pointwise (or $L^1$-$L^\infty$) local estimates. However, it is by now well-known that all the employed arguments can be extended to a more general framework where the semigroup is only assumed to have local $L^{p_-}$-$L^{p_+}$ estimates for some $p_-<2<p_+$. In such a situation the condition on the exponents $\alpha,p$ such that $A(\alpha,p)$ can be proved will depend on $p_-,p_+$. A careful examination reveals the following difference: in \cite{BCF2}, we make appear only one $\Gamma$ operator, evaluated on a product and then use a Leibniz property. In the current work, the ``carr\'e du champ'' identity \eqref{carre} makes appear the product of two $\Gamma$ operators. So combining the $\Gamma$ operator (on which we assume $L^{p_-}$-$L^{p_0}$ local estimates through $(G_{p_0})$) and the local $L^{p_-}$-$L^{p_+}$ estimates on the semigroup will then lead to more restrictions in the current setting than in \cite{BCF2}. Thus also from this point of view it is natural that we can obtain a wider range for the Sobolev algebra property, because of our stronger assumption.
\end{itemize}

As a conclusion of the comparison: our previous work \cite{BCF2} and this current one are both interesting in themselves and each of them brings results in its proper framework. If one can fit into the current framework, then it is better to follow the current approach, where we develop a simpler proof for the range $(1,p_0]$ and a wider range for $p>p_0$ by taking advantage of the carr\'e du champ identity. However \cite{BCF2} explains how we can still prove the Algebra property in a more general setting, yet with a smaller range.

\medskip

Let us now describe some examples where the extra assumption in terms of 'carr\'e du champ' identity is satisfied:
\begin{itemize}
\item The Dirichlet form setting (as detailed in \cite{BCF2}) with a carr\'e du champ. In such a case, the carr\'e du champ operator $\Gamma$ satisfies the 'strong' (pointwise) identity \eqref{eq:strong}, as well as \eqref{eq:gamma}.

\item In the Euclidean setting $M={\mathbb R}^n$ (or more generally in a doubling Riemannian manifold), consider $A = A(x)$ a complex matrix - valued function with bounded measurable coefficients, satisfying the ellipticity (or accretivity) condition
\begin{equation} \lambda |\xi|^ 2 \leq  \Re \langle A(x) \xi, \xi \rangle \qquad \textrm{and} \qquad  |\langle A(x)\xi , \zeta\rangle | \leq \Lambda |\xi||\zeta|, \label{eq:ell} \end{equation}
for some constants $\lambda,\Lambda>0$ and every $x\in {\mathbb R}^n$, $\xi,\zeta\in {\mathbb R}^n$. 

For such a complex matrix-valued function $A$, we may define a second order divergence form operator
$$ L=L_A f :=- \dive (A\nabla f),$$
which we first interpret in the sense of maximal accretive operators via a
sesquilinear form. That is, ${\mathcal D}(L)$ is the largest subspace contained in $W^{1,2}:={\mathcal D}(\nabla)$
for which 
$$ \left| \int_M \langle A \nabla f, \nabla g\rangle \, d\mu \right| \leq C \|g\|_2 \qquad \forall g \in W^{1,2},$$ 
and we define $Lf$ by
$$ \langle Lf, g \rangle = \int_M \langle A \nabla f, \nabla g\rangle \, d\mu$$
for $f \in {\mathcal D}(L)$ and $g \in W^{1,2}$. Thus defined, $L=L_A$ is a maximal-accretive operator
on $L^2$ and ${\mathcal D}(L)$ is dense in $W^{1,2}$.

For such an operator we have the pointwise carr\'e du champ identity \eqref{eq:strong} with the operator
$$ \Gamma(f,g):= \Re \langle A\nabla f, \nabla g \rangle.$$
The ellipticity condition then implies \eqref{eq:gamma}.

\item In the case of a non-selfadjoint operator $L$, we can also consider the following example: in the Euclidean space, associated with a rather singular function $a$, consider the operator $L(f) = -\Delta(af)$. It is non-selfadjoint and non conservative, but some of arguments of \cite{BCF2} or those developed here can be used, if we can prove \eqref{UE} and \eqref{eq:DG}. We refer the reader to \cite{McN} (extended to a doubling setting in \cite{DO}), where it is proven that if the measurable function $a$ has an accretive real part, then the semigroup $e^{-tL^*}$ satisfies \eqref{UE} and by duality it is also true for $e^{-tL}$. Combining this with Riesz transform estimates in $L^2$ also gives $L^2$ Davies-Gaffney estimates \eqref{eq:DG} for the operator $L$.\\
For such an operator, it is  interesting to observe that assumption \eqref{carre2} (used for \cite{BCF2}) relies on a Lipschitz condition on $a$ although the assumption \eqref{carre-w} (used here) will require a $C^2$-condition on $a$. 
\end{itemize}

\section{Technical preliminaries} 

Let us give some notation and a few reminders about certain operators constructed from the functional calculus of $L$. We refer to \cite{BCF2} for more details. We first define {\it approximation operators}, which are the elementary objects to build a paraproduct associated with a semigroup. 

\begin{definition} \label{def:Qt-Pt}
Let $N \in \N$, $N>0$, and set $c_N=\int_0^{+\infty} s^{N} e^{-s} \,\frac{ds}{s}$.
 For $t>0$, 
define
\begin{equation} \label{def:Qt}
	Q_t^{(N)}:=c_N^{-1}(tL)^{N} e^{-tL}
\end{equation}
and 
\begin{equation} \label{def:Pt}
	P_t^{(N)}:=\phi_N(tL),
\end{equation}
with $\phi_N(x):= c_N^{-1}\int_x^{+\infty} s^{N} e^{-s} \,\frac{ds}{s}$,  $x\ge 0$.
\end{definition}

Let us define some suitable sets of test functions. Let us recall that $\mathfrak{C}:=\mathcal{C}_0(M) \cap \mathcal{F}$. 

\begin{definition} \label{def:calS} For $p\in(1,{+\infty})$, we define the set of test functions 
\begin{align*}
\calS^p & =\calS^p(M,\mathcal{L}):=  \{f\in \mathfrak{C} \cap L^p: \ \exists\, g,h\in L^2 \cap L^p ,\ f=\mathcal{L}g \textrm{ and } h=\mathcal{L}f \}, 
\end{align*}
and
$$ \calS = \cup_{p\in(1,{+\infty})} \calS^p.$$
\end{definition}

We recall from \cite[Proposition 2.13]{BCF2} that \eqref{UE} implies square function estimates for $Q_t^{(N)}$ in $L^p$.
\begin{lemma} \label{lem:verticalsqfct}  Let $p \in (1,\infty)$,  $N \in \N$, $N>0$, and $\alpha>0$.
 Under \eqref{UE}, one has
$$
		\left\|\left(\int_0^\infty |(tL)^\alpha P_t^{(N)} f|^2\,\frac{dt}{t}\right)^{1/2} \right\|_p \lesssim \|f\|_p
$$
for all $f \in L^p(M,\mu)$. 
\end{lemma}

A direct consequence of the above is the following orthogonality lemma. See \cite[Lemma 2.15]{BCF2} for a slightly less general version.

\begin{lemma}\label{lem:orthogonality} Let $p \in (1,\infty)$, $N \in \N$, $N>0$, and $\alpha>0$. Assume \eqref{UE}. 
Then
$$ \left\| \int_0^{+\infty} (tL)^\alpha P_t^{(N)} F_t \, \frac{dt}{t} \right\|_{p} \lesssim \left\| \left( \int_0^{+\infty} |F_t|^2 \, \frac{dt}{t}\right)^{1/ 2} \right\|_{p},$$
where $F_t(x):=F(t,x)$,  $F: (0,{+\infty})\times M\to\R$ is a measurable function such that the RHS has a meaning and is finite.
\end{lemma}

Under the additional assumption $(G_{p_0})$ for some $p_0>2$, one also has square function estimates involving $\Gamma$. 
\begin{lemma} \label{lem:g-fct} Let $N \in \N$, $N>0$, and $\alpha \in (0,1)$. Assume $(G_{p_0})$ for some $p_0 \in (2,\infty)$.Then for every $p \in (1,p_0)$,
$$
		\left\|\left(\int_0^\infty |\sqrt{t} \Gamma (tL)^{-\alpha/2} P_t^{(N)} f|^2 \, \frac{dt}{t}\right)^{1/2} \right\|_p \lesssim \|f\|_p
$$
for all $f \in L^p(M,\mu)$. 
\end{lemma}

\begin{proof}
The proof of \cite[Proposition 2.14]{BCF2} has to be adapted as follows. 
By writing 
$$
		P_t^{(N)} f = \int_t^\infty Q_s^{(N)} f \,\frac{ds}{s},
$$
one has the pointwise estimate 
\begin{align*}
	|\sqrt{t}\Gamma (tL)^{-\alpha/2} P_t^{(N)} f|
	\leq \int_t^\infty \left(\frac{t}{s}\right)^{\frac{1-\alpha}{2}}|\sqrt{s}\Gamma (sL)^{-\alpha/2} Q_s^{(N)} f| \,\frac{ds}{s}.
\end{align*}
Since $\alpha \in (0,1)$, Hardy's inequality yields
\begin{align*}
	\left(\int_0^\infty |\sqrt{t}\Gamma (tL)^{-\alpha/2} P_t^{(N)} f|^2 \,\frac{dt}{t}\right)^{1/2} 
	\lesssim \left(\int_0^\infty |\sqrt{t}\Gamma (tL)^{-\alpha/2} Q_t^{(N)} f|^2 \,\frac{dt}{t}\right)^{1/2}. 
\end{align*}
Having this pointwise inequality, one can proceed as before in \cite[Proposition 2.14]{BCF2}. 
\end{proof}

\section{Main result}

From now on, fix $D \in \N$ in the definition of $Q_t^{(D)}$ and $P_t^{(D)}$ sufficiently large ($D >4\nu$ will suffice), and write $Q_t:= Q_t^{(D)}$ and  $P_t:=P_t^{(D)}$. \\

We define paraproducts associated with the underlying operator $L$. Note however that the definitions differ from those in \cite{BCF2}. \\

For $g \in L^\infty(M,\mu)$, we define the paraproduct $\Pi_g$ on $\mathcal{S}$ by
$$
	\Pi_g^{(D)}(f) = \Pi_g(f) := \int_0^\infty P_t (Q_tf \cdot P_tg) \,\frac{dt}{t},  \qquad f \in \mathcal{S}.
$$

For every $p \in (1,\infty)$ and every $f \in \mathcal{S}^p$, the integral is absolutely convergent in $L^p(M,\mu)$. We refer the reader to \cite[Section 3]{BCF2} for the details, noting that $(P_t)_{t>0}$ is bounded uniformly in $L^p(M,\mu)$. \\

We define the resonant term $\Pi$ on  $\mathcal{S}$ by
$$
	\Pi^{(D)}(f,g)=	\Pi(f,g):=\int_0^\infty Q_t (P_tf \cdot P_tg) \,\frac{dt}{t},  \qquad f,g \in \mathcal{S}.
$$

We discuss the question of absolute convergence of the integral in $\Pi(f,g)$ after Proposition \ref{prop:errorterms}. 

\begin{lemma}[Product decomposition]
For every $p \in (1,\infty)$ and every $f,g \in \mathcal{S}^p$, we have the product decomposition
\begin{equation}
		fg= \Pi(f,g) + \Pi_g(f) + \Pi_f(g)   \qquad \text{in}\ L^p(M,\mu).
\label{eq:dec1}
\end{equation}
\end{lemma} 

\begin{proof}
Since $\mathcal{S}^p \subseteq L^\infty(M,\mu)$, we have $f \cdot g, \, P_t f \cdot P_t g \in L^p(M,\mu)$.
We recall from \cite[Proposition 2.11, Lemma 3.1]{BCF2} that in the $L^p$ sense, $f \cdot g  = \lim_{t \to 0} P_t f \cdot P_t g$ and $0 = \lim_{t \to \infty} P_t f \cdot P_t g$, where the latter makes use of our assumption $N(L)=\{0\}$. The same arguments then also imply that
\begin{align*}
		f \cdot g & = \lim_{t \to 0} P_t(P_t f \cdot P_t g),\\
		0 &= \lim_{t \to \infty} P_t(P_t f \cdot P_t g)
\end{align*}
in the $L^p$ sense.
Since $P_t$ and $Q_t$ are defined such that $Q_t=-t\partial_t P_t$, we obtain
\begin{align*}
		fg & = \lim_{t\to 0} P_t(P_tf \cdot P_tg) - \lim_{t\to \infty} P_t(P_tf \cdot P_tg)
		= - \int_0^\infty \partial_t (P_t(P_tf \cdot P_tg))\, dt \\
		& = \int_0^\infty Q_t (P_tf \cdot P_tg) \,\frac{dt}{t}
		+ \int_0^\infty P_t (Q_tf \cdot P_tg) \,\frac{dt}{t}
		+ \int_0^\infty P_t (P_tf \cdot Q_tg) \,\frac{dt}{t},
\end{align*} 
which is the stated decomposition.
\end{proof}

The critical term in the product decomposition is the resonant term $\Pi(f,g)$. We have shown already in \cite[Proposition 3.3]{BCF2} that the paraproduct $\Pi_g(f)$ is bounded in $\dot L^p_\alpha$ for all $\alpha \in (0,1)$, without other assumption than \eqref{UE}. Let us mention that the result remains true for $\alpha \geq 1$. 

\begin{proposition} \label{prop:errorterms}
Let $p \in (1,\infty)$, $\alpha \in (0,1)$ and $g \in L^\infty(M,\mu)$. Then $\Pi_g$ is well-defined on $\mathcal{S}^p$ with for every $f \in \mathcal{S}^p$
$$
		\|\Pi_g(f)\|_{p,\alpha} \lesssim \|f\|_{p,\alpha} \|g\|_\infty.
$$
\end{proposition} 


Let us now have a look at the resonant
term $\Pi(f,g)$. We use the assumed carr\'e du champ identity \eqref{carre} to write, with $\tilde Q_t := (tL)^{-1} Q_t$,
\begin{align} \label{eq:decomp-carre} \nonumber 
		\Pi(f,g) 
		= & \int_0^\infty (tL)^{-1}Q_t tL(P_t f \cdot P_t g) \,\frac{dt}{t}  \\ \nonumber 
		 = & \int_0^\infty \tilde  Q_t(tLP_t f \cdot P_t g) \,\frac{dt}{t} 
		+ \int_0^\infty \tilde Q_t (P_tf \cdot tLP_t g) \,\frac{dt}{t} \\ 
		&- 2 \int_0^\infty \tilde Q_t\Gamma \big( \sqrt{t} P_t f , \sqrt{t}  P_t g\big) \,\frac{dt}{t}. 
\end{align}

For the first term one can use the same arguments as for $\Pi_g(f)$ to show that for $p \in (1,\infty)$, $g \in L^\infty(M,\mu)$ and $f \in \mathcal{S}^p$, the integral converges absolutely in $L^p(M,\mu)$. By interchanging the roles of $f$ and $g$, the same holds true for the second term. In the third term, for every $0<\eps<R<\infty$, the finite integral $\int_\eps^R$ is well-defined. The results of Proposition \ref{prop:lp-para} and Proposition \ref{prop:sp-para} below in particular imply that the integral converges absolutely in $L^p(M,\mu)$. \\


Instead of showing the boundedness of $\Pi(f,g)$ in $\dot L^p_\alpha$ directly, we first show its boundedness in $L^q(M,\mu)$ for large $q<\infty$, and then interpolate with $\dot L^{p_0}_1$, where $p_0$ is chosen such that $(G_{p_0})$ holds.  \\

With the same arguments as in the proof of Proposition \ref{prop:errorterms}, one immediately obtains the $L^p$ boundedness of the first term in \eqref{eq:decomp-carre}. See the proof of \cite[Proposition 3.3]{BCF2}.

\begin{lemma} \label{lem:lp-para1}
Assume \eqref{UE}. Let $p \in (1,\infty)$.  Then for every $f \in L^p(M,\mu)$ and every $g \in L^\infty(M,\mu)$, we have
$$
		\left\| \int_0^\infty \tilde  Q_t(tLP_t f \cdot P_t g) \,\frac{dt}{t} \right\|_p 
		\lesssim \|f\|_p \|g\|_\infty.
$$
\end{lemma}

For the second term, we obviously obtain the symmetric result in $f$ and $g$. But it is also possible to interchange the roles of $f$ and $g$.

\begin{lemma} \label{lem:lp-para2}
Assume \eqref{UE}. Let $p \in (1,\infty)$.  Then for every $g \in L^p(M,\mu)$ and every $f \in L^\infty(M,\mu)$, we have
$$
		\left\| \int_0^\infty \tilde  Q_t(tLP_t f \cdot P_t g) \,\frac{dt}{t} \right\|_p 
		\lesssim \|f\|_\infty \|g\|_p.
$$
\end{lemma} 

A result of this kind was  already proven in \cite[Theorem 4.2]{F}. For convenience of the reader we give a (different) proof here.

\begin{proof}
By  Lemma \ref{lem:orthogonality} applied to $T_t=\tilde Q_t$ and \cite[Theorem 2.17]{BCF2}, we have for every $q \in (p,\infty)$ - with the notation as in \cite{BCF2} -
\begin{align*}
	& \left\|\int_0^\infty \tilde Q_t(tLP_t f \cdot P_t g) \, \frac{dt}{t} \right\|_p
	\lesssim \left\| \left( \int_0^\infty |tLP_t f \cdot P_t g|^2 \,\frac{dt}{t} \right)^{1/2} \right\|_p \\
	& \lesssim \| N_\ast (P_t g)\|_p  \|\scrC_{q} (tL P_tf)\|_\infty . 
\end{align*}
We let the reader check that a simple adaptation of \cite[Lemma 4.4 (a)]{BCF2} yields $\|N_\ast (P_tg)\|_p \lesssim \|g\|_p$. Similarly, one can modify the proof of \cite[Lemma 4.4 (b)]{BCF2} for the second estimate. To do so, note that by our assumptions, 
$$
		\left\|\left(\int_0^\infty |tLP_tf|^2 \,\frac{dt}{t}\right)^{1/2} \right\|_q 
		\lesssim \|f\|_q,
$$ 
and that $(tLP_t)_{t>0}$ satisfies $L^q$ off-diagonal estimates of any order. Using this, one obtains  $\|\scrC_{q} (tL P_tf)\|_\infty \lesssim \|f\|_\infty$. 
\end{proof}

In order to treat the third term in \eqref{eq:decomp-carre}, we define the operator $\Pi_\Gamma$ on $\mathcal{S}$ by 
$$
	\Pi_\Gamma(f,g) := \int_0^\infty \tilde Q_t \Gamma \big(\sqrt{t} P_t f , \sqrt{t}  P_t g\big) \,\frac{dt}{t}, \qquad f,g \in \mathcal{S}.
$$

\begin{proposition} \label{prop:lp-para}
Assume \eqref{UE}. Let $p \in (2,\infty)$, and let $g \in L^\infty(M,\mu)$. Then $\Pi_\Gamma(\,.\,,g)$ is well-defined on $L^p(M,\mu)$ with for every $f \in L^p(M,\mu)$
$$
	\|\Pi_\Gamma(f,g) \|_p \lesssim \|f\|_p \|g\|_\infty.
$$
\end{proposition}

\begin{proof}
We can write $\tilde Q_t = (tL)^{-1} Q_t^{(D)} = [c_D^{-1}(tL)^{D-1} e^{-t/2L}]e^{-t/2L}=: \tilde{\tilde{Q}}_t P_{t/2}^{(1)}$. By Lemma \ref{lem:orthogonality} with $T_t= \tilde{\tilde{Q}}_t$ in the first step,  Minkowski's inequality in the second, \eqref{UE} and \eqref{eq:gamma} in the third, and the Cauchy-Schwarz inequality in the last step, we obtain
\begin{align} \label{prop:Lpbdd} \nonumber 
	\|\Pi_\Gamma(f,g)\|_p
	&\lesssim \left\| \left(\int_0^\infty \Big| P_{t/2}^{(1)}\Gamma \big(\sqrt{t} P_t f , \sqrt{t} P_t g\big)\Big|^2 \,\frac{dt}{t}\right)^{1/2} \right\|_p\\ \nonumber 
	& \lesssim \sum_{j=0}^\infty \left\|x \mapsto \left(\int_0^\infty \Big|P_{t/2}^{(1)}\Eins_{S_j(B(x,\sqrt{t}))} \Gamma\big(\sqrt{t} P_t f , \sqrt{t}  P_t g \big) \Big|^2 \,\frac{dt}{t}\right)^{1/2} \right\|_p\\ \nonumber 
	& \lesssim \sum_{j=0}^\infty 2^{-2jN} 2^{j \nu}  \left\| x \mapsto\left(\int_0^\infty \left(\aver{B(x,2^j\sqrt{t})} |\sqrt{t}\Gamma P_t f| \cdot |\sqrt{t} \Gamma P_t g|\,d\mu\right)^2\,\frac{dt}{t}\right)^{1/2} \right\|_p\\
	& \lesssim \sum_{j=0}^\infty 2^{-2jN} 2^{j\nu}  \left\|x \mapsto \left(\int_0^\infty \left(\aver{B(x,2^j\sqrt{t})} |\sqrt{t}\Gamma P_t f|^2 \,d\mu \right) \left(\aver{B(x,2^j\sqrt{t})} |\sqrt{t} \Gamma P_t g|^2\,d\mu\right)\,\frac{dt}{t}\right)^{1/2} \right\|_p.
\end{align}
For all $j \geq 0$ and $x \in M$, $L^2$ off-diagonal estimates for $(\sqrt{t}\Gamma P_t)_{t>0}$ (see \eqref{eq:DG})  yield
\begin{align} \label{prop:Lpbdd-eq2} \nonumber 
	\left(\aver{B(x,2^j\sqrt{t})} |\sqrt{t}\Gamma P_t g|^2\,d\mu\right)^{1/2}
	&\leq \sum_{k=0}^\infty \left(\aver{B(x,2^j\sqrt{t})} |\sqrt{t}\Gamma P_t (\Eins_{S_k(B(x,2^j\sqrt{t}))} g)|^2\,d\mu\right)^{1/2}\\ \nonumber
	& \hspace{-3cm} \lesssim  \left(\aver{B(x,2^{j}\sqrt{t})} |g|^2\,d\mu\right)^{1/2} + \sum_{k=1}^\infty \left(1+\frac{(2^{j+k}\sqrt{t})^2}{t}\right)^{-N} 2^{k\nu /2} \left(\aver{B(x,2^{j+k}\sqrt{t})} |g|^2\,d\mu\right)^{1/2}\\
	& \hspace{-3cm} \lesssim \|g\|_\infty. 
\end{align}
Using this estimate in \eqref{prop:Lpbdd}, we get
\begin{align*}
	\|\Pi_\Gamma(f,g)\|_p
	&\lesssim \|g\|_\infty \sum_{j=0}^\infty 2^{-2jN} 2^{j \nu}  \left\|x \mapsto \left(\int_0^\infty \aver{B(x,2^j\sqrt{t})} |\sqrt{t}\Gamma P_t f|^2 \,d\mu \,\frac{dt}{t}\right)^{1/2} \right\|_p\\
	& = \|g\|_\infty \sum_{j=0}^\infty 2^{-2jN} 2^{j\nu} \|\sqrt{t}\Gamma P_t f\|_{T^{p,2}_{2^j}(M)},
\end{align*}
where $T^{p,2}_{2^j}(M)$ denotes the tent space with angle $2^j$ and appropriate elliptic scaling. By change of angle in tent spaces \cite[Theorem 1.1]{angle}, $\|\sqrt{t}\Gamma P_t f\|_{T^{p,2}_{2^j}(M)} \lesssim 2^{j\nu /2} \|\sqrt{t}\Gamma P_t f\|_{T^{p,2}(M)}$ for all $p\geq 2$. On the other hand, it is known from e.g. \cite[Theorem 3.1]{AHM} (which extends to our setting) that $\sqrt{t}\Gamma P_t$ satisfies a conical square function estimate for $p \geq 2$. Thus, we finally obtain
\begin{align*}
	\|\Pi_\Gamma(f,g)\|_p 
	& \lesssim \|g\|_\infty \sum_{j=0}^\infty 2^{-2jN} 2^{j\nu} 2^{j\nu /2} \|\sqrt{t}\Gamma P_t f\|_{T^{p,2}(M)} \lesssim \|f\|_p   \|g\|_\infty.
\end{align*}
\end{proof}

Putting Lemma \ref{lem:lp-para1}, Lemma \ref{lem:lp-para2} and Proposition \ref{prop:lp-para} together, we obtain

\begin{coro}
Assume \eqref{UE}. Let $p \in (2,\infty)$, and let $g \in L^\infty(M,\mu)$. Then $\Pi(\,.\,,g)$ is well-defined on $L^p(M,\mu)$ with for every $f \in L^p(M,\mu)$
$$
		\|\Pi(f,g)\|_p \lesssim \|f\|_p \|g\|_\infty. 
$$
\end{coro}

The above result provides us with the required result at one of the endpoints in the interpolation. Let us now have a look at the other endpoint. \\

One of the terms in \eqref{eq:decomp-carre} can be estimated in $\dot L^p_\alpha$ without further assumptions. The proof is the same as the one for Proposition \ref{prop:errorterms}. 

\begin{proposition} \label{prop:error-2}
Assume \eqref{UE}. Let $p \in (1,\infty)$, $\alpha \in (0,1)$ and $g \in L^\infty(M,\mu)$. Then the integral on the left-hand side is well-defined on $\mathcal{S}^p$ with for every $f \in \mathcal{S}^p$
$$
		\left\|\int_0^\infty \tilde  Q_t(tLP_t f \cdot P_t g) \,\frac{dt}{t}  \right\|_{p,\alpha} \lesssim \|f\|_{p,\alpha} \|g\|_\infty.
$$
\end{proposition} 

The result for the resonant term can be obtained similarly to the one in Proposition \ref{prop:lp-para}, but requires the additional assumption of gradient bounds on the semigroup.

\begin{proposition} \label{prop:sp-para}
Assume \eqref{UE} and $(G_{p_0})$ for some $p_0 \in [2,\infty)$. Let $p \in (1,p_0)$, $\alpha \in (0,1)$ and $g \in L^\infty(M,\mu)$. Then the integral on the left-hand side is well-defined on $\mathcal{S}^p$ with for every $f \in \mathcal{S}^p$
$$
	\left\| \Pi_\Gamma(f,g) \right\|_{p,\alpha}	= \left\|\int_0^\infty  \tilde Q_t(\sqrt{t}\Gamma P_t f \cdot \sqrt{t} \Gamma P_t g) \,\frac{dt}{t} \right\|_{p,\alpha} \lesssim \|f\|_{p,\alpha} \|g\|_\infty.
$$
\end{proposition}


\begin{proof}
The proof is similar to the one of Proposition \ref{prop:lp-para}. We first use that by choosing $D$ in the definition of $Q_t=Q_t^{(D)}$ large enough, the operator $(tL)^{\alpha/2}Q_t$ satisfies $L^2$ off-diagonal estimates of order $N=N(D,\alpha)>\nu$. This allows to follow the steps in \eqref{prop:Lpbdd} and \eqref{prop:Lpbdd-eq2}. We obtain
\begin{align*}
	& \left\| L^{\alpha/2} \int_0^\infty Q_t (\sqrt{t}\Gamma P_t f \cdot \sqrt{t} \Gamma P_t g) \,\frac{dt}{t} \right\|_p\\
		& = \left\|(tL)^{\alpha/2} \int_0^\infty Q_t (t^{-\alpha/2} \sqrt{t}\Gamma P_t f \cdot \sqrt{t} \Gamma P_t g) \,\frac{dt}{t} \right\|_p\\
	& \lesssim \sum_{j=0}^\infty 2^{-2jN} 2^{j\nu} \left\|x \mapsto \left(\int_0^\infty \left(\aver{B(x,2^j \sqrt{t})} |t^{-\alpha/2} \sqrt{t} \Gamma P_t f|^2 \,d\mu \right) \left(\aver{B(x,2^j \sqrt{t})} |\sqrt{t}\Gamma P_t g|^2 \,d\mu\right) \,\frac{dt}{t}\right)^{1/2} \right\|_p \\
	& \lesssim  \|g\|_\infty \sum_{j=0}^\infty 2^{-2jN} 2^{j\nu} \left\| \sqrt{t} \Gamma (tL)^{-\alpha/2} P_t (L^{\alpha/2} f) \right\|_{T^{p,2}_{2^j}(M)}\\
	& \lesssim \|g\|_\infty  \| \sqrt{t} \Gamma (tL)^{-\alpha/2} P_t (L^{\alpha/2} f)\|_{T^{p,2}(M)},
\end{align*}	
where the last line follows from change of angle in tent spaces \cite[Theorem 1.1]{angle}. 
If $p \geq 2$, the above conical square function estimate is dominated by its vertical counterpart \cite[Proposition 2.1, Remark 2.2]{AHM}. Invoking Lemma \ref{lem:g-fct} for $p \in [2,p_0)$, we therefore have that the above is bounded by
\begin{align*}
	\|g\|_\infty  \| \sqrt{t} \Gamma (tL)^{-\alpha/2} P_t (L^{\alpha/2} f)\|_{L^p(M;L^2(\R_+;\frac{dt}{t}))}
	\lesssim \|g\|_\infty \|L^{\alpha/2} f\|_p. 
\end{align*}
If $p\in(1,2)$, we use \cite[Proposition 6.8]{A} (adapted to our current setting under \eqref{UE} and \eqref{R2}), to have the $L^p$-boundedness of the conical square function and we conclude to the same estimate.
\end{proof}


Stein's complex interpolation between the estimates in Proposition \ref{prop:lp-para} and Proposition \ref{prop:sp-para} on the endpoints $(\alpha,p)=(0,\infty)$ and $(\alpha,p)=(1,p_0)$ then yields

\begin{proposition} \label{prop:inter}
	Assume \eqref{UE} and $(G_{p_0})$ for some $p_0 \in [2,\infty)$. Let $p \in (p_0,\infty)$, $\alpha \in (0,\frac{p_0}{p})$ and $g \in L^\infty(M,\mu)$. Then for every  $f \in \dot L^p_\alpha(M)$, we have
$$
		\|\Pi_\Gamma(f,g)\|_{p,\alpha} \lesssim \|f\|_{p,\alpha} \|g\|_\infty.
$$
\end{proposition}

\begin{proof}
We apply Stein's complex interpolation \cite{Stein}. Let $p_1\in (p_0,\infty)$, and $\beta \in (0,1)$. 
Fix $g \in L^\infty(M,\mu)$. Define for $z \in \C$ the operator 
$$
		T_g^z :=L^{z/2} \Pi_\Gamma (L^{-z/2} \,.\,,g). 
$$
Recall that under \eqref{UE}, imaginary powers of $L$ are bounded in $L^p$ for all $p \in (1,\infty)$ (see \cite[Proposition 2.1]{BCF2}), with bound
$$
		\|L^{i\eta}\|_{p \to p} \lesssim (1+|\eta|)^s,
$$
whenever $s>\frac{\nu}{2}$ and $\eta \in \R$. 
From Proposition \ref{prop:lp-para}, we know that $T_g^0=\Pi_\Gamma(\,.\,,g)$ is a bounded operator in $L^{p_1}$. We thus obtain
$$
		\sup_{\gamma \in \R} (1+|\gamma|)^{-s} \|T_g^{i\gamma}\|_{p_1 \to p_1} \leq C^0,
$$
with $s>\frac{\nu}{2}$. On the other hand, Proposition \ref{prop:sp-para} yields that $T_g^0=\Pi_\Gamma(\,.\,,g)$  is bounded on $\dot L^p_\beta$. Hence,
$$
		\sup_{\gamma \in \R} (1+|\gamma|)^{-s} \|T_g^{\beta+i\gamma}\|_{p_0 \to p_0} \leq C^1_{\beta}. 
$$
Stein's interpolation \cite[Theorem 1]{Stein} then yields that the operator  
$$
		L^{\alpha/2}\Pi_\Gamma(L^{-\alpha/2}\,.\,,g) : L^p \to L^p
$$
is bounded whenever $\alpha= \theta\beta$ and $\frac{1}{p}=\frac{\theta}{p_0}+\frac{1-\theta}{p_1}$. Taking the limit for $\beta \to 1$ and $p_1 \to \infty$ yields the result. 

\end{proof}

\end{document}